\documentclass[12 pt]{article}
\usepackage{fullpage}
\usepackage{graphicx}
\usepackage{amsmath,amssymb,amsthm}
\usepackage{epstopdf}
\usepackage{enumitem}
\usepackage{hyperref}

\def\qed{\hfill
\ifhmode\unskip\nobreak\fi\quad\ifmmode\Box\else$\Box$\fi\\ }
\newtheorem{theorem}{Theorem}

\newtheorem{lemma}[theorem]{Lemma}

\newtheorem{claim}[theorem]{Claim}
\newtheorem{fact}[theorem]{Fact}

\newcommand{\plf}{$\mathrm{Pl}_4$}
\newcommand{\plfnof}{$\mathrm{Pl}_{4,4f}$}
\newcommand{\twg}{\mathcal{TW}}

\newcounter{claims}
\newenvironment{claims}{\refstepcounter{claims}\par\medskip\noindent%
{{\bf (\theclaims)}~~}}{\par\medskip}
\newcommand{\claimz}[2]{\begin{claims}{\em #2}\label{#1}\end{claims}}
\newcommand{\refclaim}[1]{(\ref{#1})}

\title{Planar 4-critical  graphs with four triangles}
\author{
Oleg V. Borodin \thanks{Sobolev Institute of Mathematics and Novosibirsk State University, Novosibirsk 630090, Russia.
E-mail:  {\tt  brdnoleg@math.nsc.ru}. Research of this author is supported in part  by grants 12-01-0044.
and 12-01-00631 of the Russian Foundation for Basic Research.}
\and
Zden\v{e}k Dvo\v{r}\'ak \thanks{Computer Science Institute of Charles University, Prague, Czech Republic. E-mail: {\tt rakdver@iuuk.mff.cuni.cz}. 
Supported the Center of Excellence -- Inst. for Theor. Comp. Sci., Prague (project P202/12/G061 of Czech Science Foundation), and
by project LH12095 (New combinatorial algorithms - decompositions, parameterization, efficient solutions) of Czech Ministry of Education.}
\and
Alexandr V. Kostochka \thanks{University of Illinois at Urbana-Champaign, Urbana, IL 61801, USA 
and  Sobolev Institute of Mathematics, Novosibirsk 630090, Russia. E-mail: {\tt kostochk@math.uiuc.edu}. 
Research of this author is supported in part by NSF grant DMS-0965587.}
\and
Bernard Lidick\'y \thanks{University of Illinois at Urbana-Champaign, Urbana, IL 61801, USA. E-mail: {\tt lidicky@illinois.edu}}
\and
Matthew Yancey \thanks{University of Illinois at Urbana-Champaign, Urbana, IL 61801, USA. E-mail:  {\tt yancey1@illinois.edu}. Research of this author is supported by National Science Foundation grant DMS 08-38434 ``EMSW21-MCTP: Research Experience for Graduate Students.''}
}
\date{\today}

\begin{document}
\maketitle
\begin{abstract} By the Gr\" unbaum-Aksenov Theorem (extending  Gr\" otzsch's Theorem) every  planar graph with at most three triangles is $3$-colorable. 
However, there are infinitely many planar $4$-critical graphs with exactly four triangles.
We describe all such graphs. This answers a question of Erd\H os from 1990.
\end{abstract}

\section{Introduction}
The classical Gr\" otzsch's Theorem~\cite{grotzsch1959} says that every planar triangle-free graph is $3$-colorable.
The following refinement of it is known as the Gr\" unbaum-Aksenov Theorem (the original proof of Gr\" unbaum~\cite{grunbaum1963}
was incorrect, and Aksenov~\cite{aksenov} fixed the proof).

\begin{theorem}[\cite{aksenov,borodin1997,grunbaum1963}]\label{thm-aksenov}
 Let $G$ be a planar graph containing at most three triangles. Then $G$ is 3-colorable.
\end{theorem}

The example of the complete $4$-vertex graph $K_4$ shows that ``three'' in Theorem~\ref{thm-aksenov} cannot be replaced
by ``four''. But maybe there are not many plane $4$-critical graphs with exactly four triangles ({\em \plf-graphs}, for short)? 

It turned out that there are many. Havel~\cite{havel1969} presented a  \plf-graph $H_1$  (see Figure~\ref{figH1H2}) in which the four triangles had no 
common vertices. He used the {\em quasi-edge} $H_0=H_0(u,v)$ (on the left of Figure~\ref{figH1H2}), that is, a graph in each $3$-coloring of which
the vertices $u$ and $v$ must have distinct colors. The graph $H_1$ is obtained from $K_4$ by replacing the edges $v_1u_1$ and $v_2u_2$ with copies of
the quasi-edge $H_0$.
Then Sachs~\cite{HS} in 1972 asked whether it is true that in every non-3-colorable planar graph $G$ with exactly four triangles
and no separating triangles
these triangles can be partitioned into two pairs so that in each pair the distance between the triangles is less than two.

\begin{figure}
\begin{center}
\includegraphics{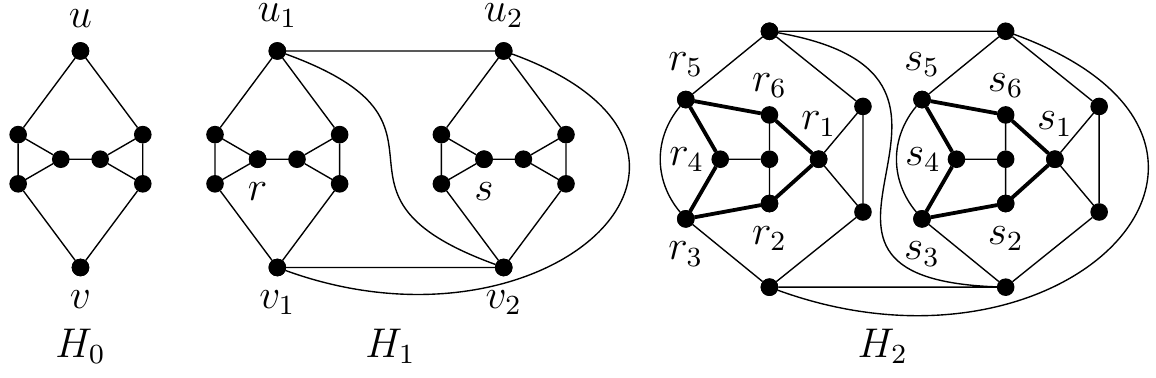}
\end{center}
\caption{A quasi-edge $H_0$ and \plf-graphs $H_1$ and $H_2$}\label{figH1H2}
\end{figure}

Aksenov and Mel'nikov~\cite{aksmel1,aksmel2} answered the question in the negative by constructing
a  \plf-graph $H_2$  (see Figure~\ref{figH1H2}) in which the 
distance between any two of the four triangles is at least two.
Moreover, they constructed two infinite series of \plf-graphs.
Aksenov~\cite{aks76} was studying \plf-graphs in the seventies.
According to Steinberg~\cite{steinberg93},
Erd\H os in 1990 asked for a description of \plf-graphs again.
Borodin~\cite{borodin1997} remarks that he knows 15 infinite families of \plf-graphs.
In his survey~\cite{borodinsurvey}, he mentions the problem of describing \plf-graphs among unsolved
problems on 3-coloring of plane graphs.
 
\begin{figure}
\begin{center}
\includegraphics{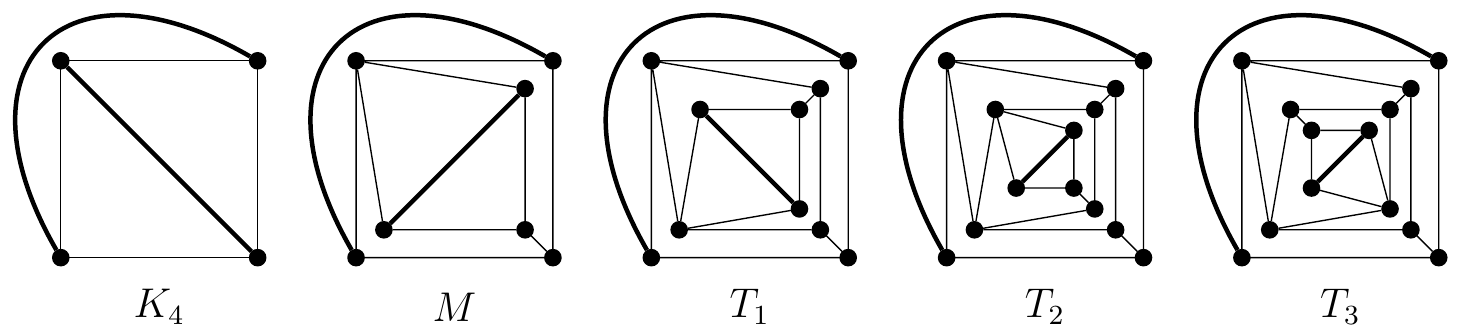}
\end{center}
\caption{Some members of $\twg$.}\label{fig-thomaswalls}
\end{figure}

First, we give a description of the \plf-graphs without $4$-faces, which we call \emph{\plfnof-graphs}.  
Thomas and Walls~\cite{thomwalls} constructed an infinite family $\twg$ of \plfnof-graphs;
the first five graphs in $\twg$ are depicted in Figure~\ref{fig-thomaswalls}
(note that $T_2$ and $T_3$ are isomorphic graphs, but their drawings are different).
If an edge $e$ of a graph belongs to exactly two triangles, we say that $e$ is a \emph{diamond} edge.
Each graph in $\twg$ contains two disjoint diamond edges, drawn in bold in Figure~\ref{fig-thomaswalls}.
We define the class $\twg$ in terms of Ore-compositions.

An \emph{Ore-composition} $O(G_1,G_2)$ of graphs $G_1$ and $G_2$ is a graph obtained as follows:
delete some edge $xy$ from $G_1$,
split some vertex $z$ of $G_2$ into two vertices $z_1$ and $z_2$ of positive degree, and
identify $x$ with $z_1$ and $y$ with $z_2$.  If $xy$ is a diamond edge of $G_1$ and $G_2$ is $K_4$,
then we say that $O(G_1,G_2)$ is a \emph{diamond expansion} of $G_1$.
The class $\twg$ consists of all graphs that can be obtained from $K_4$ by diamond expansions.

Note that $H_1$ is a \plfnof-graph but is not in $\twg$.
A graph is \emph{$k$-Ore} if it is obtained from a set of copies of $K_k$ by a sequence of Ore-compositions.
It was proved in~\cite{KY12-brooks} that every $k$-Ore graph is $k$-critical. A partial case of Theorem~6 
in~\cite{KY12-brooks} is the following.

\begin{theorem}[\cite{KY12-brooks}]\label{thm-edges}
Let $G$ be an $n$-vertex $4$-critical graph. Then $|E(G)| \geq \frac{5n-2}{3}$.
Moreover, $|E(G)| = \frac{5n-2}{3}$ if and only if $G$ is a $4$-Ore graph.
\end{theorem}

We will see below that every \plfnof-graph is a $4$-Ore graph. On the other hand,
our first result says the following.

\begin{theorem}\label{thm-or}
Every $4$-Ore graph has at least four triangles.
Moreover, a $4$-Ore graph $G$ has exactly four triangles if and only if $G$ is a \plfnof-graph.
\end{theorem}

This reduces a ``topological'' result on plane graphs to a result on abstract graphs.
And it turns out that \plfnof-graphs do not differ much from the Thomas-Walls graphs.
Let $\twg_1$ denote the graphs obtained from a graph in $\twg$ by replacing a diamond edge
by the Havel's quasi-edge $H_0$ (see Figure~\ref{fig-tw1}).  Let $\twg_2$ denote the graphs obtained from a graph in $\twg_1$ by replacing a diamond edge
by the Havel's quasi-edge $H_0$ (see Figure~\ref{fig-tw2}).  Note that $H_0$ contains no diamond edges, and thus each graph in $\twg_2$ can be obtained from a graph in
$\twg$ by replacing two vertex-disjoint diamond edges by the Havel's quasi-edge $H_0$.

\begin{figure}
\begin{center}
\includegraphics{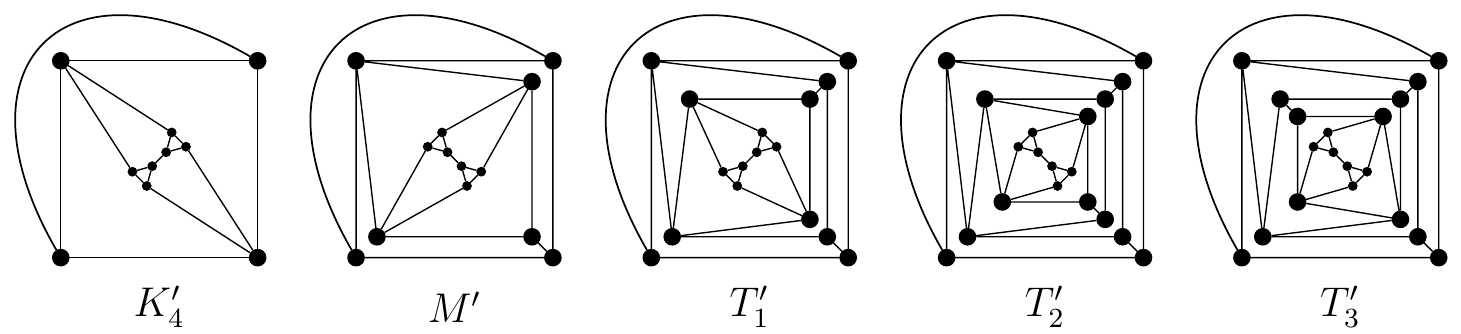}
\end{center}
\caption{Some members of $\twg_1$.}\label{fig-tw1}
\end{figure}

\begin{figure}
\begin{center}
\includegraphics{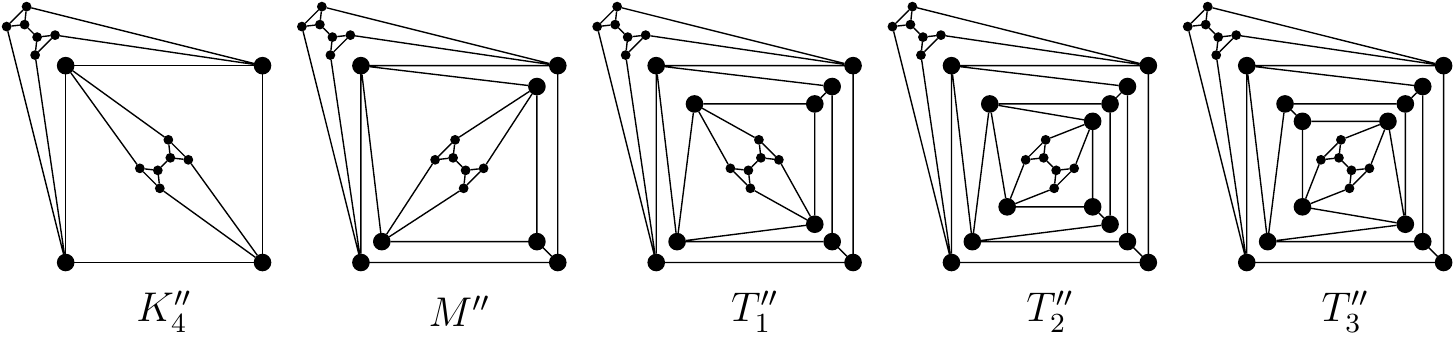}
\end{center}
\caption{Some members of $\twg_2$.}\label{fig-tw2}
\end{figure}

\begin{theorem}\label{thm-tw}
The class of \plfnof-graphs is equal to $\twg\cup\twg_1\cup \twg_2$.
\end{theorem}

The \plf-graphs may have arbitrarily many $4$-faces.
Let $G$ be a plane graph, let $C_P=xz'yx'zy'$ be a $6$-cycle in $G$
and let $\Delta$ be the closed disk bounded by $C_P$.  Let $P$ be the
subgraph of $G$ consisting of the vertices and edges drawn in $\Delta$.
If all neighbors of the vertices $x'$, $y'$ and $z'$ belong to $P$
and all faces of $G$ contained in $\Delta$ have length $4$, then we say that $P$ is a \emph{patch}.
The patch $P$ is \emph{critical} if $x'$, $y'$ and $z'$ have degree at least $3$ and
every $4$-cycle in $P$ bounds a face.
We will see that if $G_0$ is a \plf-graph and a vertex $v\in V(G_0)$ has exactly
$3$ neighbors $x$, $y$ and $z$, then the graph $G_v$ obtained from $G_0-v$ by inserting
a critical patch $P$ with boundary $C_P=xz'yx'zy'$ (where $x'$, $y'$ and $z'$ are new vertices)
is again a \plf-graph. For example, the graph $H_2$ in Figure~\ref{figH1H2} is obtained from the graph $H_1$ 
by replacing the
vertices $r$ and $s$ with patches bounded by the cycles $r_1\ldots r_6$ and
$s_1\ldots s_6$, respectively. This gives a way to construct from every \plfnof-graph an infinite
family of \plf-graphs. Our main result is that every \plf-graph can be obtained this way.

\begin{theorem}\label{thm-main}
A plane $4$-critical graph has exactly four triangles if and only if
it is obtained from a \plfnof-graph by replacing several (possibly zero)
non-adjacent $3$-vertices with critical patches.
\end{theorem}

Thus, even though there are infinitely many \plf-graphs, we know the structure of all of them.
This fully answers Erd\H os' question from 1990.
In particular, the result yields that Sachs had the right intuition in 1972: his question has positive
answer if we replace ``less than two'' with ``at most two.'' Also, Aksenov and Mel'nikov~\cite{aksmel2}
conjectured, in particular, that $H_1$ is the unique smallest \plf-graph with the minimum distance $1$
between triangles and $H_2$ is the unique smallest \plf-graph with the minimum distance $2$
between triangles. Our description confirms this.

Havel~\cite{havel1969} asked the following question:
Does there exist a constant $C$ such that every planar graph with
the minimal distance between triangles at least $C$ is 3-colorable?
The graph $H_2$ shows that $C\geq 3$, and a further example of Aksenov and Mel'nikov~\cite{aksmel2} shows that $C \geq 4$,
which is the best known lower bound.
The existence of (large) $C$ was recently proved by Dvo\v{r}\'{a}k, Kr\'al' and Thomas~\cite{dvorak09}.
It is conjectured by Borodin and Raspaud that $C=4$ is sufficient, and our result confirms this conjecture
for graphs with four triangles.  For more details, see a recent survey of Borodin~\cite{borodinsurvey}.

The proof of Theorem~\ref{thm-main} can be converted to a polynomial-time algorithm to find a $3$-coloring
of a planar graph with four triangles or to decide that no such coloring exists.  However, a more
general algorithm of Dvo\v{r}\'ak, Kr\'al' and Thomas~\cite{dvorakalg} can also be used,
and thus we do not provide further details.

The structure of the paper is as follows. In the next section we study the structure of $4$-Ore graphs.
In Section~\ref{sec-nof} we prove Theorems~\ref{thm-or} and \ref{thm-tw}. In the last section, we describe all \plf-graphs by proving 
Theorem~\ref{thm-main}.

\section{The structure of $4$-Ore graphs}
In this section, we study $4$-Ore graphs with few triangles.
The following claim is simply a reformulation of the definition of a $k$-Ore graph.

\begin{claim} \label{f1}
Every $k$-Ore graph $G\neq K_k$ has a separating set $\{x,y\}$ and two
vertex subsets $A$ and $B$ such that
\begin{itemize}[itemsep=-1mm]
\item $A\cap B=\{x,y\}$, $A\cup B=V(G)$ and no edge of $G$ connects $A\setminus\{x,y\}$
with $B\setminus\{x,y\}$,
\item $x$ and $y$ are non-adjacent in $G$ and have no common neighbor in $B$,
\item the graph $G'$ obtained from $G[A]$ by adding the edge $xy$ is a
$k$-Ore graph, and
\item the graph $G''$ obtained from $G[B]$ by identifying $x$ with $y$ into
a new vertex $x*y$ is a $k$-Ore graph.
\end{itemize}
\end{claim}

Our first goal is to obtain a similar decomposition when we restrict ourselves
to $4$-Ore graphs with $4$ triangles (Theorem~\ref{four} below).

\begin{claim} \label{g'}
Every edge in each $4$-Ore graph is contained in at most $2$ triangles.
\end{claim}
\begin{proof}
We prove the claim by induction on the order of a graph.  Let $G$ be a $4$-Ore graph
and $uv$ its edge and assume that the claim holds for all graphs with less than $|V(G)|$
vertices.

Each edge of $K_4$ is contained in exactly two triangles, and thus we can assume that $G\neq K_4$.
Let $\{x,y\}$, $A$, $B$, $G'$ and $G''$ be as in Claim~\ref{f1}.
Since $u$ is adjacent to $v$, either $\{u,v\}\subset A$ or $\{u,v\}\subset B$.
Let $G_0\in\{G[A],G[B]\}$ be the graph containing the edge $uv$, and let $G'_0\in\{G',G''\}$
be the corresponding $4$-Ore graph.  Let $u'v'$ be the edge of $G'_0$ corresponding to $uv$.
Every triangle of $G$ containing $uv$ maps to a triangle in $G'_0$ containing $u'v'$.
Since $|V(G'_0)|<|V(G)|$, the edge $u'v'$ is contained in at most two triangles in $G'_0$ by induction, and thus $uv$ is contained in at most two
triangles in $G$.
\end{proof}

For a graph $G$, let $t(G)$ denote the number of triangles in $G$.

\begin{claim} \label{g''}
If $G$ is a $4$-Ore graph, then $t(G)\ge 4$.
For every vertex $z\in V(G)$, every graph $G_z$ obtained from $G$ by splitting $z$ satisfies $t(G_z)\ge 2$.
Furthermore, if $G\neq K_4$, then $t(G-z)\ge 2$, and if $G=K_4$, then $t(G-z)=1$.
\end{claim}
\begin{proof}
Let $G$ be a $4$-Ore graph.  We proceed by induction and assume that the claim holds
for all graphs with less than $|V(G)|$ vertices.  Since the claim holds for $K_4$, we can assume that $G\neq K_4$.
Let  $\{x, y\}$, $A$, $B$, $G'$ and $G''$ be as in Claim~\ref{f1}.
By induction hypothesis and Claim~\ref{g'}, $G[A]=G'-xy$ has at
least  $t(G')-2\ge 2$ triangles.  Furthermore, $G[B]$ is obtained from $G''$ by splitting a vertex, and thus it
has at least two triangles by induction.  It follows that $G$ has at least four triangles.

Consider now a vertex $z$ of $G$.  If $z\not\in \{x,y\}$, then $G-z$ contains $G[A]$ or $G[B]$ as a subgraph,
and thus $t(G-z)\ge 2$.  If $z\in \{x,y\}$, say $z=x$, then $G-z$ contains $G'-x$ and $G''-x*y$
as vertex-disjoint subgraphs, and by induction each of them has a triangle; hence, $t(G-z)\ge 2$.
Finally, any graph $G_z$ obtained from $G$ by splitting $z$ contains $G-z$ as a subgraph, and thus $t(G_z)\ge t(G-z)\ge 2$.
\end{proof}

\begin{claim} \label{two}
If $G$ is a $4$-Ore graph with $t(G)\geq 5$, then 
$t(G-u-v)\geq 1$ for each $u,v\in V(G)$.
\end{claim}
\begin{proof}
We proceed by induction and assume that the claim holds for all graphs with less than $|V(G)|$ vertices.
Note that $G\neq K_4$, since $t(G)\geq 5$.  Let  $\{x, y\}$, $A$, $B$, $G'$ and $G''$ be as in Claim~\ref{f1}.
If $u,v\in A$, then $G''-x*y$ is a subgraph of $G-u-v$ and $t(G-u-v)\ge t(G''-x*y)\ge 1$ by Claim~\ref{g''}.
Hence, by symmetry we can assume that $u\in B\setminus \{x,y\}$.
Suppose that $v\in B$.  We can assume that $v\neq x$, and thus $G'-y$ is a subgraph of $G-u-v$.
Again, Claim~\ref{g''} implies that $t(G-u-v)\ge t(G'-y)\ge 1$.

Finally, consider the case that $u\in B\setminus\{x,y\}$ and $v\in A\setminus\{x,y\}$. Since $t(G) \geq 5$, either $G[A]$ or $G[B]$
contains at least three triangles.  In the former case, let $G_0=G[A]$, $G_0'=G'$, $r=v$ and $s=x$.  In the latter case,
let $G_0=G[B]$, $G'_0=G''$, $r=u$ and $s=x*y$.  If $r$ is contained in at most two
triangles in $G_0$, then $t(G-u-v)\ge t(G_0-r)\ge t(G_0)-2\ge 1$; hence, assume that $r$ is contained in at least $3$ triangles in $G_0$,
and thus also in $G'_0$. Note that $G'_0\neq K_4$, since $t(G_0)\ge 3$.  By Claim~\ref{g''}, we have $t(G'_0)\ge t(G'_0-r)+3\ge 5$.
Note that $G'_0-r-s$ is a subgraph of $G-u-v$, and thus $t(G-u-v)\ge t(G'_0-r-s)\ge 1$ by induction.
\end{proof}

\begin{claim} \label{three}
If $G$ is a $4$-Ore graph with $t(G)\geq 5$, then 
$t(G-v)\geq 3$ for each $v\in V(G)$.
\end{claim}

\begin{proof}
Note that $G\neq K_4$, since $t(G)\geq 5$.  Let  $\{x, y\}$, $A$, $B$, $G'$ and $G''$ be as in Claim~\ref{f1}.
Since $t(G) \geq 5$, either $G[A]$ or $G[B]$ contains at least three triangles.  In the former case,
let $G_0=G[A]$, $G_1=G[B]$, $G_0'=G'$ and $s=x$.  In the latter case, let $G_0=G[B]$, $G_1=G[A]$, $G'_0=G''$ and $s=x*y$.
Since $G_0$ has at least three triangles, it follows that $G_0'\neq K_4$.  
If $v\not\in V(G_0)$, then $t(G-v)\ge t(G_0)\ge 3$.  Therefore, assume that $v\in V(G_0)$.
If $v\in\{x,y\}$, then $G'-v$ and $G''-x*y$
are vertex-disjoint subgraphs of $G-v$, and since at least one of $G'$ and $G''$ is not equal to $K_4$, we have
$t(G-v)\ge t(G'-v)+t(G''-x*y)\ge 3$ by Claim~\ref{g''}.

Finally, suppose that $v\not\in\{x,y\}$, and thus $G_1$ is a subgraph of $G-v$.  
By Claim~\ref{g''}, there are at least two triangles in $G_1$.  We claim that $G_0-v$ contains a triangle.
This is clear if $v$ belongs to at most two triangles in $G_0$, since $t(G_0)\ge 3$.  Otherwise, $v$ belongs to at least three triangles in $G_0$, and thus also
in $G'_0$.  By Claim~\ref{g''}, we have $t(G_0')\ge t(G'_0-v)+3\ge 5$.  We conclude that
$t(G_0-v)\ge t(G_0'-v-s)\ge 1$ by Claim~\ref{two}.  Therefore, $t(G-v)=t(G_0-v)+t(G_1)\ge 3$.
\end{proof}

A \emph{$4,4$-Ore graph} is a $4$-Ore graph with exactly $4$ triangles.
The main result of this section is the following.

\begin{theorem} \label{four}
Suppose $G$ is a $4,4$-Ore graph distinct from $K_4$.
Let $\{x, y\}$, $A$, $B$, $G'$ and $G''$ be as in Claim~\ref{f1}.
Then both $G'$ and $G''$ are $4,4$-Ore graphs,
$xy$ is a diamond edge of $G'$, and $t(G[B])=2$.  Furthermore,
if $G''\neq K_4$, then $x*y$ belongs to exactly two triangles of $G''$.
\end{theorem}
\begin{proof}
Note that $t(G[A])\ge 2$ and $t(G[B])\ge 2$ by Claims~\ref{g'} and~\ref{g''},
and since $t(G[A])+t(G[B])=t(G)=4$, it follows that $t(G[A])=t(G[B])=2$.
If $t(G')\ge 5$, then we would have $t(G[A])\ge t(G'-x)\ge 3$ by Claim~\ref{three}.
If $t(G'')\ge 5$, then we would have $t(G[B])\ge t(G''-x*y)\ge 3$ by Claim~\ref{three}.
It follows that $t(G')\le 4$ and $t(G'')\le 4$, and by
Claim~\ref{g''}, we conclude that $t(G')=t(G'')=4$, i.e., both $G'$ and $G''$
are $4,4$-Ore graphs.

Since $xy$ belongs to $t(G')-t(G[A])=2$ triangles in $G'$, it is a diamond edge.
Similarly, $x*y$ belongs to at least $t(G'')-t(G[B])=2$ triangles in $G''$.
Furthermore, if $G''\neq K_4$ and $x*y$ belonged to at least three triangles,
then we would have $t(G'')\ge t(G''-x*y)+3\ge 5$ by Claim~\ref{g''}, which is a contradiction.
\end{proof}

\section{A description of $4,4$-Ore graphs and \plfnof-graphs}\label{sec-nof}

With Theorem~\ref{four}, it is easy to characterize all $4,4$-Ore graphs.
The \emph{Moser spindle} is the Ore composition of two $K_4$'s, depicted in Figure~\ref{fig-thomaswalls} as $M$.

\begin{lemma}\label{lemma-44ore}
Every $4,4$-Ore graph belongs to $\twg\cup\twg_1\cup \twg_2$.
\end{lemma}
\begin{proof}
Let $G$ be a $4,4$-Ore graph.  We proceed by induction and assume that every $4,4$-Ore graph with less than
$|V(G)|$ vertices belongs to $\twg\cup\twg_1\cup \twg_2$.  Note that $K_4\in \twg$, and thus we can
assume that $G\neq K_4$.  Let $\{x, y\}$, $A$, $B$, $G'$ and $G''$ be as in Theorem~\ref{four}.
By induction hypothesis, we have $G',G''\in \twg\cup\twg_1\cup \twg_2$.
Since $G'$ has a diamond edge $xy$ and $G''$ has a vertex $x*y$ belonging to at least two triangles,
we conclude that $G',G''\not\in \twg_2$.  Since $G[B]$ has exactly two triangles,
an inspection of the graphs in $\twg\cup \twg_1$ (see Figures~\ref{fig-thomaswalls} and
\ref{fig-tw1}) shows that either
\begin{itemize}
\item $G''$ is the Moser spindle, $x*y$ is its vertex of degree four, $G[B]$ is the Havel's quasiedge $H_0$
and $x$ and $y$ are its vertices of degree two, or
\item $x$ has degree two in $G[B]$, $y$ has degree one in $G[B]$ and $x*y$ is incident with a diamond edge $(x*y)z$ of $G''$,
where $z$ is the neighbor of $y$ in $G[B]$.
\end{itemize}
In the former case, $G$ is obtained from $G_1$ by replacing a diamond edge with the Havel's quasiedge $H_0$,
and thus $G\in \twg_1\cup \twg_2$.  In the latter case, observe that the described Ore-composition of graphs
from $\twg\cup \twg_1$ results in a graph from $\twg\cup\twg_1\cup \twg_2$.
\end{proof}

We can now describe \plfnof-graphs.
\begin{proof}[Proof of Theorems~\ref{thm-tw} and~\ref{thm-or}]
By Claim~\ref{g''}, every $4$-Ore graph has at least four triangles.  Therefore, it suffices to prove
that the following classes are equal to each other:
\begin{itemize}
\item $4,4$-Ore graphs,
\item $\twg\cup\twg_1\cup\twg_2$, and
\item \plfnof-graphs.
\end{itemize}

By Lemma~\ref{lemma-44ore}, every $4,4$-Ore graph belongs to $\twg\cup\twg_1\cup\twg_2$.  Every graph
in $\twg\cup\twg_1\cup\twg_2$ is planar, $4$-critical (since it is $4$-Ore) and has four triangles, and thus it is a \plfnof-graph.
Therefore, it suffices to show that every \plfnof-graph $G$ is $4,4$-Ore.  Since $G$ has four triangles,
we only need to prove that $G$ is $4$-Ore.  Consider a plane drawing of $G$ without $4$-faces.  Let $e$, $n$ and $s$
be the number of edges, vertices and faces of the drawing of $G$, respectively.  Note that
$G$ has at most four triangular faces and all other faces of $G$ have length at least $5$.  Therefore,
$G$ has at least $\frac{1}{2}(5(s-4)+3\cdot 4)$ edges.
Note that $G$ is connected (since it is $4$-critical), and thus $s=e+2-n$ by Euler's formula.
It follows that $e\le (5n-2)/3$, and $G$ is $4$-Ore by Theorem~\ref{thm-edges}.
\end{proof}

\section{A description of \plf-graphs}

We are going to characterize planar $4$-critical graphs with $4$ triangles.
To deal with short separating cycles, we use the notion of criticality with respect to a subgraph.
Let $G$ be a graph and $C$ be its (not necessarily induced) proper subgraph.
We say that $G$ is \emph{$C$-critical (for $3$-coloring)} if for every proper subgraph $H\subset G$
such that $C\subseteq H$, there exists a $3$-coloring of $C$ that extends
to a $3$-coloring of $H$, but not to a $3$-coloring of $G$.

Notice that $4$-critical graphs are exactly $C$-critical graphs with $C=\emptyset$.
Furthermore, it is easy to see that if $F$ is a $4$-critical graph and $F=G\cup G'$, where $C=G\cap G'$,
then either $G=C$ or $G$ is $C$-critical.  We mainly use the following reformulation.

\begin{lemma}[Dvo\v{r}\'ak et al.~\cite{trfree1}]\label{lemma-critin}
Let $G$ be a plane graph and let $\Lambda$ be a connected open region of the plane whose boundary is equal
to a cycle $C$ of $G$, such that $\Lambda$ is not a face of $G$.
Let $H$ be the subgraph of $G$ drawn in the closure of $\Lambda$.
If $G$ is $4$-critical, then $H$ is $C$-critical.
\end{lemma}

This is useful in connection with the following result of Gimbel and Thomassen~\cite{gimbelthomassen},
which was also obtained independently by Aksenov et al.~\cite{aksenov6cyc}.

\begin{theorem}[Gimbel and Thomassen~\cite{gimbelthomassen}]\label{thm-gimbel}
Let $G$ be a plane triangle-free graph with the outer face bounded by a cycle $C$ of length at most $6$.  If $G$ is $C$-critical,
then $C$ is a $6$-cycle and all internal faces of $G$ have length four.
\end{theorem}

Furthermore, they also exactly characterized the colorings of $C$ that do not extend to $G$.
\begin{theorem}[Gimbel and Thomassen~\cite{gimbelthomassen}]\label{thm-gimbelext}
Let $G$ be a plane graph with the outer face bounded by a cycle $C=c_1\ldots c_6$ of length $6$, such that all other
faces of $G$ have length $4$.  A $3$-coloring $\varphi$ of $G[V(C)]$
does not extend to a $3$-coloring of $G$ if and only if $\varphi(c_1)=\varphi(c_4)$, $\varphi(c_2)=\varphi(c_5)$ and $\varphi(c_3)=\varphi(c_6)$.
\end{theorem}

\begin{figure}
\center{\includegraphics{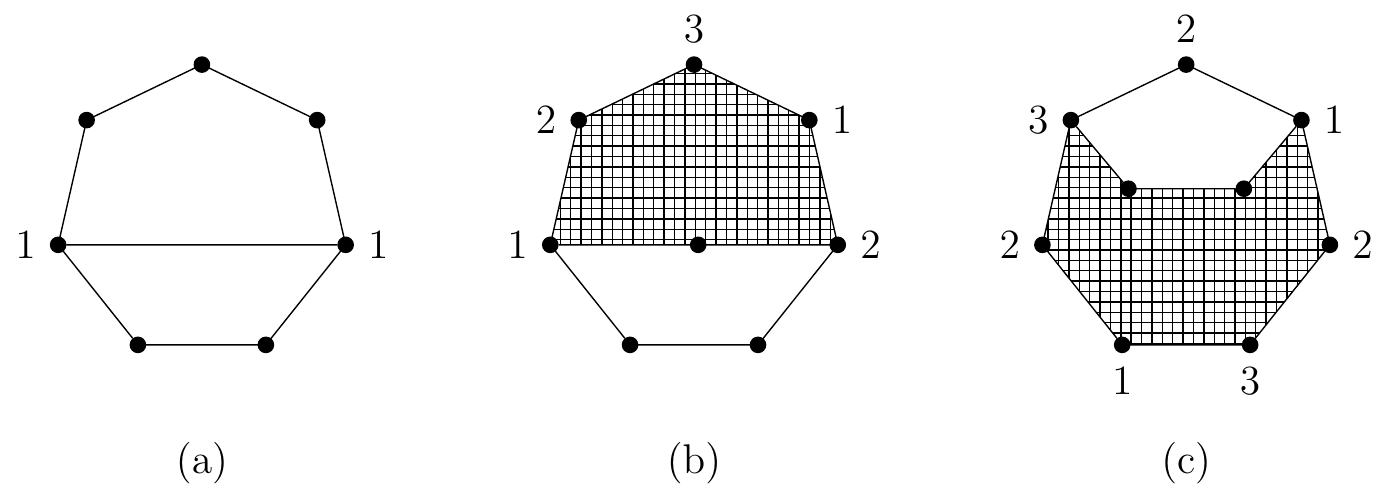}}
\caption{Critical graphs with a precolored $7$-face.}\label{fig-7cyc}
\end{figure}
An analogous result for a $7$-cycle was obtained by Aksenov et al.~\cite{aksenov7cyc}.
Their paper is in Russian. For a shorter proof in English see~\cite{dvorak8cyc}.

\begin{theorem}[Aksenov et al.~\cite{aksenov7cyc}]\label{thm-7cyc}
Let $G$ be a plane triangle-free graph with the outer face bounded by a cycle $C=c_1\ldots c_7$ of length $7$.  The graph $G$ is $C$-critical
and $\varphi:V(C)\to \{1,2,3\}$ is a $3$-coloring of $C$ that does not extend to a $3$-coloring of $G$ if and only if $G$ contains no separating cycles of length at most five
and one of the following propositions is satisfied up to relabelling of vertices (see Figure~\ref{fig-7cyc} for an illustration).
\begin{itemize}
\item[\textrm(a)] The graph $G$ consists of $C$ and the edge $c_1c_5$, and $\varphi(c_1)=\varphi(c_5)$.
\item[\textrm(b)] The graph $G$ contains a vertex $v$ adjacent to $c_1$ and $c_4$, the cycle $c_1c_2c_3c_4v$ bounds a 5-face and every face drawn inside
the $6$-cycle $vc_4c_5c_6c_7c_1$ has length four; furthermore, $\varphi(c_4)=\varphi(c_7)$ and $\varphi(c_5)=\varphi(c_1)$.
\item[\textrm(c)] The graph $G$ contains a path $c_1uvc_3$ with $u,v\not\in V(C)$, the cycle $c_1c_2c_3vu$ bounds a 5-face and every face drawn inside
the $8$-cycle $uvc_3c_4c_5c_6c_7c_1$ has length four; furthermore, $\varphi(c_3)=\varphi(c_6)$, $\varphi(c_2)=\varphi(c_4)=\varphi(c_7)$
and $\varphi(c_1)=\varphi(c_5)$.
\end{itemize}
\end{theorem}

By inspection of the cases (a),(b) and (c) of Theorem~\ref{thm-7cyc} we obtain the following fact.
\begin{fact}\label{f-5f}
Let $G$ be a plane triangle-free graph with the outer face bounded by a cycle $C$ of length $7$.
Suppose that $G$ is $C$-critical and that $\varphi$ is a precoloring of $C$ that does not extend to a $3$-coloring of $G$.
Let $x$, $y$ and $z$ be consecutive vertices of $C$.  If $\varphi(x)=\varphi(z)$, then $y$ is incident with the $5$-face of $G$.
\end{fact}

We also use the following result dealing with graphs with a triangle.

\begin{theorem}[Aksenov~\cite{aksenov}]\label{thm-aksen}
Let $G$ be a plane graph with the outer face bounded by a cycle $C=c_1c_2\ldots$ of length at most $5$.  If $G$ is $C$-critical
and contains exactly one triangle $T$ distinct from $C$, then $C$ is a $5$-cycle,
all internal faces of $G$ other than $T$ have length exactly four and $T$ shares at least one edge with $C$.
Furthermore, if $C$ and $T$ share only one edge $c_1c_2$ and $\varphi$ is a $3$-coloring of $C$
that does not extend to a $3$-coloring of $G$, then $\varphi(c_1)=\varphi(c_3)$ and $\varphi(c_2)=\varphi(c_5)$.
\end{theorem}

Recall that the notion of a patch was defined in the introduction.  First, we need to argue that replacing vertices
of degree  $3$  by critical patches preserves criticality and the number of triangles.
\begin{lemma}\label{lemma-pcr}
Let $G$ be a $4$-critical graph, let $v$ be a vertex of $G$ of degree  $3$  and let $x$, $y$ and $z$ be the neighbors of $v$ in $G$.
Let $G'$ be a graph obtained from $G-v$ by inserting a patch $P$ with boundary $C_P=xz'yx'zy'$, where $x'$, $y'$ and $z'$ are new vertices.
Then $G'$ is $4$-critical if and only if the patch $P$ is critical.
Furthermore, if $P$ is critical, then $t(G')=t(G)$.
\end{lemma}
\begin{proof}
Clearly, every $4$-critical graph has minimum degree at least $3$.  Furthermore, by Lemma~\ref{lemma-critin} and Theorem~\ref{thm-gimbel},
if $\Lambda$ is an open disk bounded by a $4$-cycle in a $4$-critical plane graph and no triangle is contained in $\Lambda$, then $\Lambda$ is
a face.  Therefore, if $G'$ is $4$-critical, then $P$ is a critical patch.

Suppose now conversely that $P$ is a critical patch.  Consider an edge $e\in E(G')$.  If $e\not\in E(P)$, then $e$ is an edge
of $G$ not incident with $v$.  Since $G$ is $4$-critical, there exists a $3$-coloring $\varphi$ of $G-e$.  The vertex $v$ is properly colored,
and thus we can by symmetry assume that $\varphi(x)=\varphi(y)$.  By Theorem~\ref{thm-gimbelext}, $\varphi$ extends to a $3$-coloring of $P$,
showing that $G'-e$ is $3$-colorable.

On the other hand,  consider the case that $e\in E(P)$.  Since $G$ is $4$-critical, there exists a $3$-coloring $\psi$ of $G-v$ such that
$\psi(x)=1$, $\psi(y)=2$ and $\psi(z)=3$.  Suppose that $\varphi$ does not extend to a $3$-coloring of $P-e$.
If $e\not\in E(C_P)$, this implies that $P-e$ contains a $C_P$-critical subgraph $P'$.  By Theorem~\ref{thm-gimbel}, all faces of $P'$ distinct
from $C_P$ have length $4$.  However, $e$ is drawn inside one of the faces of $P'$, and thus $P$ would contain a $4$-cycle not bounding a face,
contrary to the assumption that $P$ is a critical patch.

Finally, suppose that $e\in E(C_P)$, say $e=xz'$.
Note that all faces of $P$ have even length, and thus $P$ is bipartite.  Since $z'$ has degree at least
 $3$  and every $4$-cycle in $P$ bounds a face, we conclude that $xz'y$ is the only path of length at most two between $x$ and $y$ in $P$.
Let $P_1$ be obtained from $P-e$ by adding the edge $xy$, and note that $P_1$ is triangle-free.
Let $C_1=xyx'zy'$ be the $5$-cycle bounding the outer face of $P_1$.
Since $\varphi$ does not extend to a $3$-coloring of $P-e$, it also does not extend to a $3$-coloring of $P_1$, and
thus $P_1$ contains a $C_1$-critical subgraph.  This contradicts Theorem~\ref{thm-gimbel}.

We conclude that $G'-e$ is $3$-colorable for every $e\in E(G')$.  Since $G'$ does not contain isolated vertices,
this implies that every proper subgraph of $G'$ is $3$-colorable.  Suppose that $G'$ has a proper $3$-coloring $\theta$.
Since $G$ is not $3$-colorable, $\theta$ cannot be extended to $v$, and thus we can assume that $\theta(x)=1$, $\theta(y)=2$
and $\theta(z)=3$.  However, that implies that $\theta(x)=\theta(x')$, $\theta(y)=\theta(y')$ and $\theta(z)=\theta(z')$.
Since $\theta$ is a $3$-coloring of $P$, this contradicts Theorem~\ref{thm-gimbelext}.  Therefore, $G'$ is not $3$-colorable,
and thus it is $4$-critical.

Now we establish a bijection $f$ between triangles in $G$ and $G'$.  If a triangle $T$ in $G$ does not contain $v$, then $T$ also appears in $G'$,
and we set $f(T)=T$.  If $T$ contains $v$, say $T=vxy$, then we set $f(T)=z'xy$.  Since $f$ is injective,  it suffices to show that
it is surjective.  Suppose that there exists a triangle $T'\subset G'$ that is not in the image of $f$.  Then $T'$ contains
an edge of $P$.  Since $P$ is bipartite, $T'$  contains an edge outside of $P$, and since $x$, $y$ and $z$ are non-adjacent in the patch $P$, we conclude that $T'$
intersects $P$ in a path of length two, say $xwy$, and $x$ and $y$ are adjacent in $G$.  Since $f(vxy)=xz'y$ and $T'$ is not in the image
of $f$, we conclude that $w\neq z'$.  Since $P$ is a critical patch, the $4$-cycle $xz'yw$ bounds a face.  However, this implies that $z'$ has
degree two, which is a contradiction.  Therefore, $f$ is indeed a bijection, and thus $t(G')=t(G)$.
\end{proof}

By Lemma~\ref{lemma-pcr}, the graphs described in Theorem~\ref{thm-main} are indeed $4$-critical and have exactly $4$ triangles.
If $G$ is obtained from a \plfnof-graph by replacing non-adjacent vertices of degree $3$ with (not necessarily critical) patches, then we say that
$G$ is an \emph{expanded \plfnof-graph}.  By Lemma~\ref{lemma-pcr},  it remains to show that every \plf-graph is an expanded \plfnof-graph.
First we state several properties of expanded \plfnof-graphs.

\begin{claim}\label{cl-patch}
If $P$ is a patch in an expanded \plfnof-graph $G$, then no vertex of $P$ is incident with exactly one edge that does not belong to $P$.
\end{claim}
\begin{proof}
Let $G_0$ be the \plfnof-graph from which $G$ is obtained by replacing vertices with patches.  Let $v$ be the vertex of $G_0$ that
is replaced by $P$ and let $x$, $y$ and $z$ be the neighbors of $v$.  All vertices of $V(P)\setminus \{x,y,z\}$ only have
neighbors in $P$.  Each of $x$, $y$ and $z$ has degree at least  $3$  in $G_0$, and thus each of them is incident with
at least two edges that are not incident with $v$.  So, each of $x$, $y$ and $z$ is incident with at least two edges
of $G$ that do not belong to $P$.
\end{proof}

\begin{claim}\label{cl-facelen}
Every face of an expanded \plfnof-graph $G$ has length $3$, $4$ or $5$, and every $4$-face of $G$ belongs to a patch.
\end{claim}
\begin{proof}
Let $G_0$ be the \plfnof-graph from which $G$ is obtained by replacing vertices with patches.
By Theorem~\ref{thm-tw}, $G_0$ belongs to $\twg\cup\twg_1\cup\twg_2$, and thus each face of $G_0$
has length $3$ or $5$.  Furthermore, observe that replacing a $3$-vertex $v$ by a patch transforms
each face of $G_0$ incident with $v$ to a face of $G$ of the same length, whose boundary
shares a path of length two with the boundary cycle of $P$.  Therefore, every face of $G$ which
is not contained in a patch has length $3$ or $5$.
\end{proof}

Furthermore, we will use the following simple property of critical graphs.
\begin{claim}\label{cl-septri}
If $G$ is a $4$-critical graph and $T$ is a triangle of $G$, then $G-V(T)$ is connected.
In particular, if $G$ is a plane graph, then every triangle in $G$ bounds a face.
\end{claim}
\begin{proof}
If $G-V(T)$ is not connected, then there exist proper subgraphs $G_1$ and $G_2$ of $G$
such that $G_1\cup G_2=G$ and $G_1\cap G_2=T$.  Since $G$ is $4$-critical, both $G_1$ and $G_2$ are $3$-colorable
and by permuting the colors if necessary, we can assume that their $3$-colorings match on $T$.
Together, they would give a $3$-coloring of $G$, which is a contradiction.
\end{proof}

A \emph{stretching} of a plane graph $G$ at a vertex $w\in V(G)$ is a graph $G_2$ obtained from $G$ by the following procedure.
Let $e_1,\ldots, e_k$ be the edges incident with $w$ as drawn in the clockwise order around it.
Choose $m<k$ and let $G'$ be obtained from $G$ by removing $w$, adding two new vertices $w_1$ and $w_2$
and adding edges between $w_1$ and the endpoints of $e_1$, \ldots, $e_m$ distinct from $w$, and between
$w_2$ and the endpoints of $e_{m+1}$, \ldots, $e_k$ distinct from $w$.  Let $G_1$ be obtained from $G'$
by either adding a new vertex $z$ adjacent to $w_1$ and $w_2$, or by adding an edge between $w_2$ and the
endpoint $z$ of $e_1$ distinct from $w$.  Finally, for each face $f$ of $G_1$ incident with $zw_2$,
if $|f|=6$, then replace $f$ by a quadrangulation, and if $|f|=7$, then replace it by a graph satisfying (a), (b) or (c)
of Theorem~\ref{thm-7cyc}, resulting in the graph $G_2$.  We call $G_1$ the \emph{intermediate graph} of the stretching, and the faces of $G_1$
incident with $zw_2$ are called \emph{special}.

\begin{lemma}\label{lemma-main}
Every \plf-graph is an expanded \plfnof-graph.
\end{lemma}
\begin{proof}
Assume that $G$ has the fewest vertices among the
 \plf-graphs that are not  expanded \plfnof-graphs.
  Then $G$ does not contain any patches, since replacing
a patch by a 3-vertex would give a smaller counterexample.  We are going to need the following stronger claim.

\claimz{cl-no6}{Let $C$ be a $6$-cycle in $G$ and let $\Delta$ be an open region of the plane bounded by $C$, such that all faces in $\Delta$ have length four.
Then $\Delta$ contains at most one vertex, and if it contains a vertex, then each vertex of $C$ is incident with an edge that is not drawn in the closure of $\Delta$.}
\begin{proof}
Let $C=v_1v_2v_3v_4v_5v_6$ and suppose that $\Delta$ contains at least one vertex, and if it contains only one, then all edges incident with $v_6$ are drawn in the closure of $\Delta$.
Let $G'$ be the graph obtained from $G$ by removing the vertices in $\Delta$ (which does not include $C$) and adding a vertex $v$ adjacent to $v_1$, $v_3$ and $v_5$.  By Theorem~\ref{thm-gimbelext},
$G'$ is not $3$-colorable, since every $3$-coloring of $C$ in which $v_1$, $v_3$ and $v_5$ do not have pairwise distinct colors can be extended to the subgraph of $G$ drawn in the closure of $\Delta$.
Therefore, $G'$ has a $4$-critical subgraph $G''$.  Note that $|V(G'')|<|V(G)|$, since 
if $\Delta$ contains only one vertex of $G$, then $v_6$ has degree $2$ in $G'$, and thus $v_6\not\in V(G'')$.

Since both $G$ and $G''$ are $4$-critical, $G''$ is not a proper subgraph of $G$.
Hence $v$ and all edges incident to it belong to $G''$.
 Triangles in $G''$ that are not in $G$ can only be created by adding  $v$ to an edge, say $v_1v_3$; in this case, $v_1v_2v_3$ is a triangle in $G$ which bounds
a face  (by Claim~\ref{cl-septri}), and thus $v_2$ has degree two in $G'$ and does not belong to $G''$, which means that creating a new triangle
in $G''$ destroys another triangle.
Thus  
$t(G'')\le t(G)=4$. 
Since 
 $G''$ is not $3$-colorable, it contains exactly four triangles. 
By the minimality of $G$, the graph $G''$ is an expanded \plfnof-graph and by Claim~\ref{cl-facelen}, each face of $G''$ has length at most $5$.
Let $G_1$ be the graph obtained from $G''-v$ by adding the subgraph $H$ of $G$ contained in the closure of $\Delta$.
Then $G_1$ is a subgraph of $G$.

If at least two faces of $G''$ incident with $v$ have length $4$, then $v$ is a vertex of a patch
$P''$ of $G''$ and all neighbors of $v$ belong to $P''$.
We claim that $P=(P''-v)+H$ is a patch; this is clear if $v$ is incident with three $4$-faces in $G''$.  If $v$ is incident with exactly two $4$-faces, then let $f$ be the face incident with $v$
of length other than four.  Note that at least one edge $e$ of $C$ does not belong to $G''$ and in $G$, it is drawn in the region of the plane corresponding to $f$.  By symmetry, we can assume that $e=v_1v_6$.
Since all faces of $G''$ have length at most $5$, if $v_6\in V(G'')$, then $v_1$ and $v_6$ would be joined by a path $Q$ of length two in $G''$.
The path $Q$ together with the edge $v_1v_6$ would form a triangle in $G$ which does not correspond to any triangle in $G''$, and thus $G''$ would only have at most three triangles, which is a contradiction.
Therefore, $v_6\not\in V(G'')$, and thus all neighbors of $v_6$ in $G_1$ belong to $P$, showing that $P$ is a patch.
Consequently, $G_1$ is an expanded \plfnof-graph.

Hence, we can assume that at most one face of $G''$ incident with $v$ has length four.  If exactly one face incident with $v$ had
length $4$, then $v$ would belong to a patch $P$ and it would be incident with exactly one edge not belonging to $P$, contradicting Claim~\ref{cl-patch},
Therefore, no face incident with $v$ has length four, and as in the previous paragraph, we conclude that $H$ is a patch in $G_1$ that replaces the vertex $v$ of degree  $3$  of $G''$.
Again, it follows that $G_1$ is an expanded \plfnof-graph.

Since $G_1$ is not $3$-colorable and  $G$ is $4$-critical,  $G=G_1$ and $G$ is an expanded \plfnof-graph.  This is a contradiction.
\end{proof}

\claimz{cl-4cycle}{$G$ does not contain separating 4-cycles.}
\begin{proof}
Suppose that $C=v_1v_2v_3v_4$ is a separating $4$-cycle in $G$.  Let $\Lambda_1$ and $\Lambda_2$ be the
two connected regions obtained from the plane by removing $C$.  For $i\in\{1,2\}$, let $G_i$ be the
subgraph of $G$ drawn in the closure of $\Lambda_i$.  By Lemma~\ref{lemma-critin}, $G_i$ is $C$-critical,
and by Theorem~\ref{thm-aksen}, $G_i$ contains at least two triangles.  Since $t(G)=4$, we
conclude that $t(G_1)=t(G_2)=2$.

For $i,j\in\{1,2\}$, let $G_{i,j}$ be the graph obtained from $G_i$ by adding the edge $v_jv_{j+2}$.
Note that $C+v_jv_{j+2}$ has a unique $3$-coloring up to permutation of colors.  Therefore,
any $3$-colorings of $G_{1,j}$ and $G_{2,j}$ could be combined to a $3$-coloring of $G$.
We conclude that at least one of $G_{1,j}$ and $G_{2,j}$ is not $3$-colorable.
By symmetry, we can assume that $G_{1,1}$ is not $3$-colorable.

Let $G'_{1,1}$ be a $4$-critical
subgraph of $G_{1,1}$.  By Claim~\ref{cl-septri}, each triangle in $G'_{1,1}$ bounds a face,
and thus $v_1v_3$ belongs to at most two triangles of $G'_{1,1}$.  Since $t(G_1)=2$, it follows
that $t(G'_{1,1})\le 4$.  Since $G'_{1,1}$ is $4$-critical, it follows that $t(G'_{1,1})=4$, i.e., $G'_{1,1}$ is a \plf-graph,
and that $v_1v_3$ belongs to two triangles of $G'_{1,1}$.  Let $v_1v'_2v_3$ and $v_1v'_4v_3$ be the triangles
incident with $v_1v_3$ labelled so that for $k\in\{2,4\}$, $v_k$ is either equal to $v'_k$ or
it is drawn in $G_{1,1}$ in the region of the plane corresponding to the face of $G'_{1,1}$ bounded by $v_1v'_kv_2$.

Since $C$ is a separating cycle, $|V(G'_{1,1})|<|V(G)|$, and by the minimality of $G$,
it follows that $G'_{1,1}$ is an expanded \plfnof-graph.  By Claim~\ref{cl-facelen},
every face of $G'_{1,1}$ has length at most five.  Consider a face $f$ of $G'_{1,1}$ not incident with $v_1v_3$.
Since the subgraph of $G$ drawn in the closure of $f$ contains no triangles,  by Lemma~\ref{lemma-critin} and Theorem~\ref{thm-gimbel}, 
 $f$ is a face of $G$ as well.  Furthermore, if $v_k\neq v'_k$ for some $k\in\{2,4\}$, then
Lemma~\ref{lemma-critin} and Theorem~\ref{thm-gimbel} imply that $v_1v'_kv_3v_k$
is a face of $G$ and $v_k$ has degree two in $G_{1,1}$.  It follows that $V(G_{1,1})\setminus V(G'_{1,1})\subseteq\{v_2,v_4\}$
and each vertex $v\in V(G_{1,1})\setminus V(G'_{1,1})$ is  adjacent  in $G_{1,1}$ only to $v_1$ and $v_3$.

Suppose that $G'_{1,1}$ has a $4$-face.  In that case, it contains a patch $P$.  Let $x$, $y$ and $z$ be the vertices of the boundary cycle $K$
of $P$ that have neighbors only in $P$.  Note that $v_1v_3\not\in E(K)$, since both faces incident
with $v_1v_3$ are triangles.  Furthermore, at most two of $x$, $y$ and $z$ are incident with $C$,
since they form an independent set.  Therefore, we can assume that all edges incident with $x$ in $G$ belong to $P$.
This contradicts (\ref{cl-no6}).  We conclude that $G'_{1,1}$ has no $4$-faces, and thus
it belongs to $\twg\cup\twg_1\cup\twg_2$ by Theorem~\ref{thm-tw}.

Consequently, both vertices of the diamond edge $v_1v_3$ of $G'_{1,1}$ have degree exactly  $3$  in $G'_{1,1}$.
Since $G'_{1,1}$ is $4$-critical,  there exists a $3$-coloring $\varphi$ of $G'_{1,1}-\{v_1,v_3\}$
and  $\varphi(v'_2)\neq \varphi(v'_4)$.  Set $\varphi(v_k)=\varphi(v'_k)$ for $k\in\{2,4\}$ and note that
$\varphi$ is a proper $3$-coloring of $G_{1,2}$.  Since $G$ is not $3$-colorable, it follows that $G_{2,2}$
is not $3$-colorable.  By an argument symmetrical to the one for $G_{1,1}$, we conclude that
$G_{2,2}$ contains a \plfnof-graph $G'_{2,2}$ as a subgraph such that
$V(G_{2,2})\setminus V(G'_{2,2})\subseteq\{v_1,v_3\}$
and each vertex $v\in V(G_{2,2})\setminus V(G'_{2,2})$ is only adjacent to $v_2$ and $v_4$ in $G_{2,2}$.

Suppose that $G'_{1,1}\neq G_{1,1}$, and thus say $v_2\in V(G_{1,1})\setminus V(G'_{1,1})$ is  adjacent  in $G_1$
only to $v_1$ and $v_3$.
Since $v_2v_4$ is a diamond edge of $G'_{2,2}$, vertex $v_2$ has degree  $3$  in $G'_{2,2}$.
If $G'_{2,2}=G_{2,2}$, this would imply that $v_1$ and $v_3$ are the only neighbors of $v_2$ in $G_2$,
and thus $v_2$ would have degree two in $G$, contrary to the assumption that $G$ is $4$-critical.
Hence, we can assume that say $v_1\in V(G_{2,2})\setminus V(G'_{2,2})$ is  adjacent  in $G_2$ only  to $v_2$ and $v_4$.
Since $G'_{1,1}$ and $G'_{2,2}$ are $4$-critical, there exist $3$-colorings $\varphi_1$ of $G'_{1,1}-v_1v_3$
and $\varphi_2$ of $G'_{2,2}-v_2v_4$ such that $\varphi_1(v_1)=\varphi_1(v_3)=1$, $\varphi_1(v'_2)=2$, $\varphi_1(v'_4)=3$,
$\varphi_2(v_2)=\varphi_2(v_4)=3$, $\varphi_2(v'_1)=2$ and $\varphi_2(v'_3)=1$.  Then each vertex of $G$ is colored
by $\varphi_1$ or $\varphi_2$ and  if a vertex ($v_3$ or $v_4$) belongs to both $G'_{1,1}$ and $G'_{2,2}$,
then it is assigned the same color by $\varphi_1$ and $\varphi_2$.  Thus the union of $\varphi_1$ and $\varphi_2$
gives a $3$-coloring of $G$, which is a contradiction.

Therefore, $G'_{1,1}=G_{1,1}$, and by symmetry, $G'_{2,2}=G_{2,2}$.  It follows that $G$ has no $4$-face,
and thus $G$ is a \plfnof-graph.  This is a contradiction.
\end{proof}

\claimz{cl-5cyc}{If $K$ is a separating $5$-cycle in $G$ and $\Delta$ is a region of the plane bounded by $K$ such that
$\Delta$ contains at most one triangle, then $\Delta$ contains exactly one vertex, which has  three  neighbors in $K$.}
\begin{proof}
By Theorem~\ref{thm-gimbel} and the criticality of $G$,  $\Delta$ contains a triangle $T$,
as $\Delta$ is not a face of $G$.  Since $K$ is separating, Theorem~\ref{thm-aksen} implies that
$T$ shares exactly one edge with $K$.  Let $\{z\}=V(T)\setminus V(K)$.  Note that $T\cup K$ contains a $6$-cycle,
bounding an open region $\Delta_0\subset\Delta$.  Since all edges incident with $z$ are drawn in the closure of $\Delta_0$,
\refclaim{cl-no6} implies that $\Delta_0$ contains no vertices, and thus $z$ is the only vertex in $\Delta$.
\end{proof}

\claimz{cl-nodiam}{The following configuration does not appear in $G$: a path $z_1z_2z_3$ and a vertex $z$ of degree  $3$  adjacent
to $z_1$, $z_2$ and $z_3$.}
\begin{proof}
Since $G$ does not contain separating triangles, if $z_1$ is adjacent to $z_3$, then  $G=K_4$.  This is a contradiction,
since $G$ is not a \plfnof-graph.

Let $G'$ be the graph obtained from $G-z$ by identifying $z_1$ with $z_3$ to a new vertex $u$.   Since $G$ is not $3$-colorable,
$G'$ also is not $3$-colorable.  Let $G''$ be a $4$-critical subgraph of $G'$.  Note that $G''$ has at least four triangles,
and since the triangles $zz_1z_2$ and $zz_2z_3$ disappear during the construction of $G'$, we conclude that $G''$ contains
at least two triangles $ux_1x_2$ and $uy_1y_2$ such that $z_1x_1x_2z_3$ and $z_1y_1y_2z_3$ are (not necessarily disjoint)
paths between $z_1$ and $z_3$ in $G-\{z,z_2\}$.  Since $G''$ contains no separating triangles by Claim~\ref{cl-septri},
we conclude that $G''$ contains exactly two such triangles, and thus $t(G'')=4$.  Let $\Delta_1$ be the open disk
bounded by $z_1zz_3x_2x_1$ in $G$ corresponding to the face $ux_1x_2$ of $G'$, and let $\Delta_2$ be the open disk
bounded by $z_1z_2z_3y_2y_1$ in $G$ corresponding to the face $uy_1y_2$ of $G'$.  By swapping the labels of $x_i$ and $y_i$
(for $i\in\{1,2\}$) if necessary, we can assume that $\Delta_1$ and $\Delta_2$ are disjoint.  Note that since $t(G'')=4$,
neither $\Delta_1$ nor $\Delta_2$ contains a triangle of $G$, and thus $\Delta_1$ and $\Delta_2$ are faces of $G$ by
\refclaim{cl-5cyc}.

By the minimality of $G$, it follows that $G''$ is an expanded \plfnof-graph, and all faces of $G''$ have length at most $5$.
By Lemma~\ref{lemma-critin} and Theorem~\ref{thm-gimbel}, all faces of $G''$ other than $ux_1x_2$ and $uy_1y_2$ are also faces of $G$.
In particular, since $G$ does not contain patches, $G''$ also does not contain patches, and thus $G''$ has no $4$-faces.
It follows that $G$ has no $4$-faces.  Therefore, $G$ is a \plfnof-graph, which is a contradiction.
\end{proof}

\claimz{cl-nohavel}{The following configuration does not appear in $G$: a triangle $T=z_1z_2z_3$ such that all vertices of $T$ have degree  $3$ 
and $z_3$ is adjacent to a vertex $x_3$ distinct from $z_1$ and $z_2$ that has degree  $3$  and belongs to a triangle.}
\begin{proof}
Let $x_1$ and $x_2$ be the neighbors of $z_1$ and  $z_2$, respectively, outside of $T$.  By~\refclaim{cl-nodiam}, we have $x_1\neq x_2\neq x_3\neq x_1$.
Furthermore, the vertices $x_1$, $x_2$ and $x_3$ form an independent set in $G$, as otherwise every $3$-coloring of $G-V(T)$ would extend
to $G$, contrary to the $4$-criticality of $G$.  Let $G'$ be the graph obtained from $G-V(T)-x_3$ by adding the edge $x_1x_2$.
Note that every $3$-coloring of $G'$ extends to a $3$-coloring of $G$, and thus $G'$ is not $3$-colorable.  Let $G''$ be a $4$-critical
subgraph of $G$.  Note that $G''$ contains at least four triangles, and since $T$ as well as the triangle incident with $x_3$
disappear during the construction of $G'$, it follows that $x_1x_2$ belongs to at least two triangles in $G''$.  By Claim~\ref{cl-septri},
$x_1x_2$ belongs to exactly two triangles, each of them bounding a face of $G''$.  Therefore, $G$ contains a $4$-cycle $x_1ux_2v$ separating
two of its triangles from $T$.  By \refclaim{cl-4cycle}, it follows that $u$ is adjacent to $v$.  By \refclaim{cl-5cyc},
either $x_1z_1z_2x_2u$ or $x_1z_1z_2x_2v$ bounds a face, and thus $u$ or $v$ has degree  $3$ .  This contradicts~\refclaim{cl-nodiam}.
\end{proof}

Let $v_1v_2v_3v_4$ be a $4$-face in $G$, if possible chosen so that it contains two adjacent vertices that are only
incident with $4$-faces.  Since $G$ contains no separating triangles, $v_1v_3$ and $v_2v_4$ are not edges.
Suppose that $v_1$ and $v_3$ are joined by a path $v_1x_1x_2v_3\subset G-\{v_2,v_4\}$ of length  $3$ , and that $v_2$ and $v_4$
are joined by a path $v_2y_1y_2v_4\subset G-\{v_1,v_3\}$ of length  $3$ .  By symmetry and planarity, we can assume that $x_1=y_1$.
If both $v_1$ and $v_2$ have degree $3$, then every $3$-coloring of $G-\{v_1,v_2\}$ extends
to $G$, contrary to the assumption that $G$ is $4$-critical.  By symmetry, $v_1$ has degree at least $4$, and
since $v_4y_2y_1v_1$ is not a separating $4$-cycle, $v_1$ is adjacent to $y_2$.  By~\refclaim{cl-5cyc} applied to the $5$-cycle $v_4v_1v_2y_1y_2$,
we have $x_2=y_2$.  But then $V(G)=\{v_1,v_2,v_3,v_4,x_1,x_2\}$ and $G$ is $3$-colorable.

Therefore, we can by symmetry assume that $v_1$ and $v_3$ are not joined by a path of length  $3$  in $G-\{v_2,v_4\}$.  Let $G'$ be the graph obtained
from $G$ by identifying $v_1$ and $v_3$ to a single vertex $w$.  Clearly, $G'$ has exactly the four triangles that originally belonged to $G$ (possibly
with $v_1$ or $v_3$ relabelled to $w$).
Since every $3$-coloring of $G'$ gives a $3$-coloring of $G$, we conclude that $G'$ is not $3$-colorable.  Let $G''$ be a $4$-critical subgraph of $G'$.
By the minimality of $G$,  $G''$ is an expanded \plfnof-graph.
In particular, all faces of $G''$ have length at most $5$.  For a face $f$ of $G''$, let $C_f$ denote the corresponding cycle in $G$
(either equal to $f$ up to relabelling of $w$ to $v_1$ or $v_3$, or obtained from $f$ by replacing $w$ by the path $v_1v_2v_3$).  Since all triangles of $G$ are faces of $G''$,
Theorem~\ref{thm-gimbel} implies that if $|C_f|=|f|$, then $C_f$ is a face of $G$.  Furthermore, if $|C_f|=|f|+2$, then $4\le |f|\le 5$
and $C_f$ does not bound a face of $G$, since $v_2$ has degree at least  $3$ .  Let $G_f$ denote the subgraph of $G$ drawn in the region of the
plane bounded by $C_f$ and corresponding to $f$, and note that $G_f$ is one of the graphs described by Theorem~\ref{thm-gimbel} or Theorem~\ref{thm-7cyc}.
Therefore,  the following holds.

\claimz{cl-stretch}{The graph $G$ is obtained by stretching from an expanded \plfnof-graph $G_0$ at a vertex $w_0$, where $G$ and $G_0$ have the same triangles.}
In particular, all faces of $G$ have length at most $5$.
By the choice of the $4$-face $v_1v_2v_3v_4$, we can also assume that the following property is satisfied.
\claimz{cl-all4}{Let $G_0$ and $w_0$ be as in~\refclaim{cl-stretch}.  Let $v_1v_2v_3$ be the path replacing $w_0$ in the intermediate graph
of the stretching. The path $v_1v_2v_3$ is contained in the boundary of a $4$-face in $G$.  Furthermore,
if $G$ contains an edge whose vertices are only incident with $4$-faces, then $v_1$ is only incident with $4$-faces.}

Let $G_0$ and $w_0$ satisfying~\refclaim{cl-stretch} be chosen so that if $G_0$ contains at least one patch, then~\refclaim{cl-all4} is satisfied,
and subject to that with $|V(G_0)|$ as small as possible.  Observe that $G_0$ 
does not necessarily have to be created from $G$ by contracting a $4$-face.
Let $G_1$ be the intermediate graph of the stretching, let $f_1$ and $h_1$ be its special faces and let $f_0$ and $h_0$ be the corresponding
faces of $G_0$.

\claimz{cl-no4}{The graph $G_0$ has no $4$-faces.}
\begin{proof}
Suppose  that $G_0$ has a $4$-face, and thus it contains a patch $P_0$ bounded by a $6$-cycle
$C_0=z_1z_2z_3z_4z_5z_6$, where $z_2$, $z_4$ and $z_6$ only have neighbors in $P_0$.  Let $r_2$, $r_4$ and $r_6$ be the faces of $G_0$ incident
with $z_2$, $z_4$ and $z_6$, respectively, whose length is not four.  Since $G_0$ is $2$-connected, these faces are pairwise distinct.
If $C_0$ is a $6$-cycle in $G$, then the open region $\Delta$ corresponding to $P_0$
contains only $4$-faces and at least one vertex of $G$ lies inside $\Delta$.  
There exists $k\in\{2,4,6\}$ such that $f_0\neq r_k\neq h_0$, and thus 
all edges incident with $z_k$ in $G$ are drawn in the closure of $\Delta$.
This contradicts~\refclaim{cl-no6}; hence, $C_0$ contains $w_0$ and corresponds to an $8$-cycle $C_1$ in $G_1$ containing the
path $v_1v_2v_3$.  Note that all faces of $G$ drawn in the open region $\Lambda$ of the plane bounded by $C_1$ and corresponding to $P_0$
have length~$4$.  
By symmetry, we can assume that $w_0\in \{z_1,z_2\}$ and $f_0$ is contained in the patch $P_0$.

Suppose that $h_0$ does not share an edge with $C_0$; then, $w_0=z_1$, $r_2$ and $r_6$ are faces of $G$ and $v_1$ is incident with one of them.
Since neither $r_2$ nor $r_6$ is a $4$-face,~\refclaim{cl-all4} implies that $G$ does not contain an edge whose vertices are only incident with $4$-faces.
Hence, $|V(P_0)\setminus V(C_0)|=1$ and the only neighbors of $z_1$ in $P_0$ are $z_2$ and $z_6$.  By~\refclaim{cl-all4}, $G$ contains a $4$-face
$v_1v_2v_3z$.  If $z$ lies inside $\Lambda$, then note that it is adjacent neither to $z_2$ nor to $z_6$, and thus $z$ has a neighbor inside $\Lambda$.
If $z$ lies outside $\Lambda$, then similarly $v_2$ has a neighbor inside $\Lambda$.  In both cases $G$ contains an edge whose vertices are incident
only with $4$-faces,  a contradiction.  Therefore, $h_0$ shares an edge with $C_0$ (we can assume that $h_0=r_2$).
Since the same conclusions would have to hold for any other patch in $G_0$,   the only patch in $G_0$ is  $P_0$.

Label the vertices of $C_1$ so that the vertices of $V(C_0)\setminus \{w_0\}$ retain their labels and $C_1=z_1xyz_2z_3z_4z_5z_6$.
Note that all edges incident with $y$, $z_2$, $z_4$ and $z_6$ in $G_1$ are drawn in the closure of $\Lambda$.
Also, either $x$ is adjacent to $z_1$ in $G_0$ and $|h_1|=|h_0|=5$, or $|h_0|=5$ and $|h_1|=7$ and all edges incident with $x$ are drawn in the closure of $\Lambda$.
Let $h_0=z_1z_2z_3ab$, where $b$ may be equal to $x$. Since $C=z_1baz_3z_4z_5z_6$ is a $7$-cycle in $G$ whose interior (the part of the plane bounded by $C$ and containing $\Lambda$)
does not contain any triangle,  it satisfies (b) or (c) of Theorem~\ref{thm-7cyc}.
Let $G_0'$ be the graph obtained from $G_0$ by replacing the patch $P_0$ with a $3$-vertex $z'$ adjacent to $z_1$, $z_3$ and $z_5$.
Since $G_0$ has only one patch, $G'_0$ has no patches and it is a \plfnof-graph.
Furthermore, $G$ is obtained from $G'_0$ by stretching at vertex $z'$ (split out the edge $z'z_3$, then make its new endvertex adjacent to $z_5$ and relabel
the vertices created from $z'$ by $z_4$ and $z_6$).  
Since $|V(G'_0)| < |V(G_0)|$, this is a contradiction with the choice of $G_0$.
\end{proof}

\begin{figure}
\center{\includegraphics{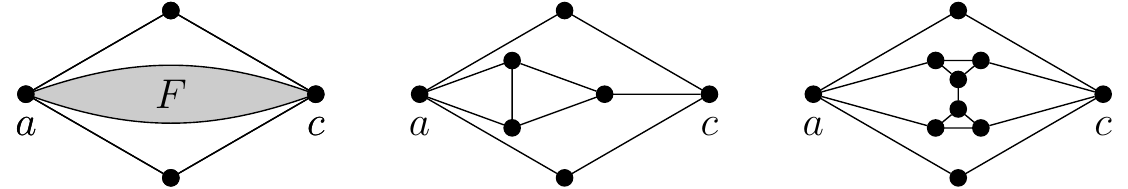}}
\caption{Possibilities for a fragment $F$.}\label{fig-fragment}
\end{figure}

Suppose that $\{a,c\}$ is a vertex-cut in $G_0$ and $F$ is an induced subgraph
of $G_0$ such that all edges of $G_0$ incident with $V(F)\setminus\{a,c\}$ belong to $F$.
We say that $F$ is a \emph{fragment with attachments $a$ and $c$} if $F$ either isomorphic
to the Havel's quasiedge $H_0$ and $a$ and $c$ are its vertices of degree two, or
if $F$ consists of a $4$-cycle $au_1u_2u_3$ and edges $u_1u_3$ and $u_2c$.
See Figure~\ref{fig-fragment} for an illustration.
Since all triangles of $G_0$ belong to $G$ and stretching does not preserve size of at least one face, $G_0$ is not $K_4$,
and thus $G_0$ contains two fragments.

\begin{figure}
\center{\includegraphics{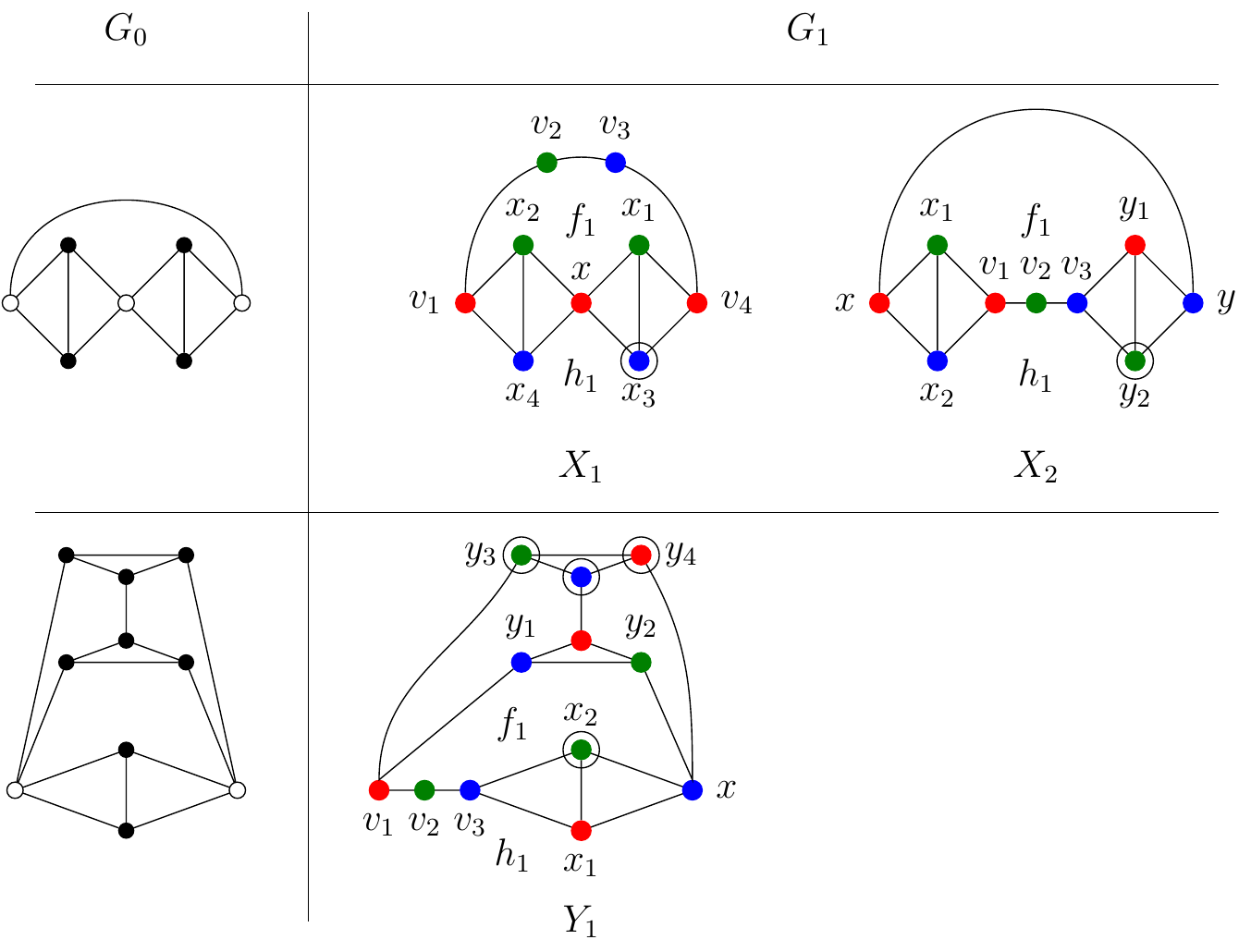}}
\caption{Stretchable \plfnof-graphs.}\label{fig-pk}
\end{figure}

Since stretching has at most two special faces and fragment
contains configurations forbidden by~\refclaim{cl-nodiam} and~\refclaim{cl-nohavel},
the stretching must occur at a vertex of each of the fragments. Hence, $G_0$ does not
have two vertex-disjoint fragments, and we conclude that $G$ is one of the graphs
$M$ and $K'_4$ depicted in Figures~\ref{fig-thomaswalls} and \ref{fig-tw1}.
The possible intermediate graphs $G_1$ are drawn in Figure~\ref{fig-pk}. Thus we have the three cases below.

Case 1: $G_1=X_1$. Let $\psi$ be a $3$-coloring  of $G_1$
such that $\psi(v_1)=\psi(x)=\psi(v_4)=1$, $\psi(v_2)=\psi(x_1)=\psi(x_2)=2$
and $\psi(v_3)=\psi(x_3)=\psi(x_4)=3$.  This coloring does not extend to
a $3$-coloring of $G$, and by symmetry, we can assume that it does not
extend to a $3$-coloring of $G_{h_0}$.  By Fact~\ref{f-5f},
it follows that $x$, $x_3$ and $v_4$ are incident with a common $5$-face in $G$,
and thus $x_3$ has degree  $3$ .  This contradicts~\refclaim{cl-nodiam}.

Case 2: $G_1=X_2$.  Let $\psi$ be a  $3$-coloring  of $G_1$
such that $\psi(v_1)=\psi(x)=\psi(y_1)=1$, $\psi(v_2)=\psi(x_1)=\psi(y_2)=2$
and $\psi(v_3)=\psi(x_2)=\psi(y)=3$.  By symmetry, $\psi$ does not extend to $G_{h_0}$,
and by Theorem~\ref{thm-7cyc}, $x_2$ is adjacent to $v_3$ and $xx_2v_3y_2y$ is a $5$-face.
However, then $y_2$ has degree  $3$  and we again obtain a contradiction with~\refclaim{cl-nodiam}.

Case 3: $G_1=Y_1$.  Let $\psi$ be a $3$-coloring  of $G_1$
such that $\psi(v_1)=\psi(x_1)=\psi(y_4)=1$, $\psi(v_2)=\psi(y_2)=\psi(x_2)=\psi(y_3)=2$ and
$\psi(v_3)=\psi(x)=\psi(y_1)=3$.
If $\psi$ does not extend to a $3$-coloring of $G_{f_0}$, then by
Fact~\ref{f-5f}, $G$ contains a $5$-face incident with $v_3$, $x_2$ and $x$; hence,
$x_2$ has degree $3$ and contradicts~\refclaim{cl-nodiam}.  Since $G$ is not $3$-colorable,
it follows that $\psi$ does not extend to a $3$-coloring of $G_{h_0}$.
By Fact~\ref{f-5f}, $G$ contains a $5$-face incident with $v_1$, $y_3$ and $x$, and
by Theorem~\ref{thm-gimbel}, $y_4$ is incident with the $5$-face as well.
However, then both $y_3$ and $y_4$ have degree  $3$  in $G$, which contradicts~\refclaim{cl-nohavel}.

This finishes the proof of Lemma~\ref{lemma-main}.
\end{proof}

\begin{proof}[Proof of Theorem~\ref{thm-main}]
By Lemma~\ref{lemma-pcr}, every graph obtained from a \plfnof-graph by replacing
non-adjacent $3$-vertices with critical patches is $4$-critical and has exactly four triangles.
Conversely, if $G$ is a \plf-graph, then it is an expanded \plfnof-graph by Lemma~\ref{lemma-main},
and since $G$ is $4$-critical, all the patches in $G$ are critical by Lemma~\ref{lemma-pcr}.
\end{proof}

\bibliographystyle{acm}
\bibliography{4c4t}

\end{document}